\documentclass[amstex,12pt]{amsart}
\usepackage{mathrsfs}
\usepackage{amssymb}
\usepackage{amsfonts}
\usepackage{amsfonts,amsmath,amsthm, amssymb}
\usepackage{array}
\usepackage{latexsym, euscript, epic, eepic}
\usepackage{tikz}
\usetikzlibrary{matrix}

\def\XXint#1#2#3{{\setbox0=\hbox{$#1{#2#3}{\int}$}
\vcenter{\hbox{$#2#3$}}\kern-.5\wd0}}

\newtheorem{theorem}{Theorem}[section]
\newtheorem{lemma}[theorem]{Lemma}
\newtheorem{proposition}[theorem]{Proposition}
\newtheorem{corollary}[theorem]{Corollary}
\newtheorem{remark}[theorem]{Remark}

\begin{document}

\title[Symmetric homeomorphisms on the real line]{Symmetric and strongly symmetric homeomorphisms on the real line
with non-symmetric inversion} 

\author[H. Wei]{Huaying Wei} 
\address{Department of Mathematics and Statistics, Jiangsu Normal University \endgraf Xuzhou 221116, PR China} 
\email{hywei@jsnu.edu.cn} 

\author[K. Matsuzaki]{Katsuhiko Matsuzaki}
\address{Department of Mathematics, School of Education, Waseda University \endgraf
Shinjuku, Tokyo 169-8050, Japan}
\email{matsuzak@waseda.jp}

\subjclass[2010]{Primary 30C62; Secondary 30F60}
\keywords{quasisymmetric, symmetric homeomorphism, strongly quasisymmetric, BMO, VMO}
\thanks{Research supported by the National Natural Science Foundation of China (Grant No. 11501259)
and Japan Society for the Promotion of Science (KAKENHI 18H01125).}

\begin{abstract}
We show an example of a symmetric homeomorphism $h$ of the real line $\mathbb{R}$ 
onto itself such that $h^{-1}$ is not symmetric. This implies that
the set of all symmetric self-homeomorphisms of $\mathbb{R}$ does not constitute a group under the composition. We 
also deal with strongly symmetric self-homeomorphisms of $\mathbb{R}$ along the same line. 
These results reveal the difference of the sets of such self-homeomorphisms of the real line from those of the unit circle.
\end{abstract}

\maketitle

\section{Introduction and statement of the result}
An increasing homeomorphism $h$ of the real line $\mathbb{R}$ onto itself is said to be 
{\it quasisymmetric} 
(or $M$-quasisymmetric to specify the constant) if there exists a constant $M \geq 1$ such that 
$$
\frac{1}{M} \leq \frac{h(x+t)-h(x)}{h(x)-h(x-t)}\leq M
$$
for all $x\in\mathbb{R}$ and $t>0$. 
The ratio in the middle term is called the quasisymmetry quotient of $h$ and is denoted by $m_h(x,t)$.  The optimal value of such $M$ is called the quasisymmetry constant of $h$. Beurling and Ahlfors \cite{BA} proved that $h$ is quasisymmetric if and only if there exists some quasiconformal homeomorphism of the upper half-plane $\mathbb{U}=\{x+iy \in \mathbb{C} \mid y>0\}$ onto itself that is continuously extendable to the boundary map $h$. 
Let $\rm QS(\mathbb{R})$ denote the group of all quasisymmetric homeomorphisms of the real line $\mathbb{R}$.

A quasisymmetric homeomorphism $h$ is said to be {\it symmetric} if 
$$ 
\lim_{t\to 0}\frac{h(x+t)-h(x)}{h(x)-h(x-t)}=1 
$$
uniformly for all $x\in\mathbb{R}$.
Let $\rm S(\mathbb{R})$ denote the subset
of $\rm QS(\mathbb{R})$ consisting of all symmetric homeomorphisms of the real line $\mathbb{R}$. It is known that $h$ is symmetric if and only if $h$ can be extended to an asymptotically conformal homeomorphism $f$ 
of the upper half-plane $\mathbb{U}$ onto itself (see \cite{Ca, GS}). 
In fact, the Beurling--Ahlfors extension of $h$ is asymptotically conformal when $h$ is symmetric. 
By an {\it asymptotically conformal} homeomorphism $f$ of the upper half-plane $\mathbb{U}$, we mean that its complex dilatation $\mu = \Bar{\partial}f/\partial f$ satisfies that
$$
{\rm ess}\!\!\!\!\sup_{y \leq t\qquad}\!\!\!\!|\mu(x+iy)| \to 0 \quad (t \to 0).
$$

We consider mappings on the unit circle $\mathbb{S}$ by using those on $\mathbb R$.
Let $g:\mathbb{S}\to\mathbb{S}$ be an orientation-preserving self-homeomorphism of $\mathbb{S}$. 
We take the lift $\hat{g}$ of $g$ under the universal cover $u:\mathbb{R}\to\mathbb{S}$ given by $u(x)=e^{2\pi ix}$, 
that is, $\hat{g}:\mathbb{R}\to\mathbb{R}$ is the uniquely determined continuous function with $u\circ \hat{g}= g\circ u$ up to additive constants. Clearly, $\hat{g}$ satisfies $\hat{g}(x+1)=\hat{g}(x)+1$.

By taking the lift,
we can define $g: \mathbb{S}\to\mathbb{S}$ to be {\it quasisymmetric} if $\hat{g}$ is quasisymmetric, while $g$ to be {\it symmetric} if $\hat{g}$ is  symmetric (see \cite{GS, Ma}). We denote the set of all quasisymmetric homeomorphisms of $\mathbb{S}$ by $\rm QS$, and the set of all symmetric homeomorphisms of $\mathbb{S}$ by $\rm Sym$. We see that $g$ is quasisymmetric if and only if $g$ can be extended to a quasiconformal homeomorphism $f$ of the unit disk $\mathbb{D}$ onto itself. We also see that $g$ is symmetric if and only if $g$ can be extended to an {\it asymptotically conformal} homeomorphism $f$ of $\mathbb{D}$ onto itself in the sense that its complex dilatation 
$\mu = \bar{\partial}f/\partial f$ satisfies that
$$
{\rm ess}\!\!\!\!\!\sup_{|z| \geq 1-t\quad}\!\!\!\!\!|\mu(z)| \to 0 \quad (t \to 0).
$$
This result is attributed to Fehlmann \cite{Fe} in \cite{GS}.
Each element $g \in \rm Sym$ satisfying a certain normalization condition
becomes a point in the little Teichm\"uller space $T_0$ (see \cite{GL}). 

By the chain rule of complex dilatations, the composition of asymptotically conformal homeomorphisms of $\mathbb{D}$
and the inverse of an asymptotically conformal homeomorphism of $\mathbb{D}$ are also asymptotically conformal. Consequently, $\rm Sym$ is a subgroup of $\rm QS$. Moreover, it was proved by Gardiner and Sullivan \cite{GS} that $\rm Sym$ is the characteristic topological subgroup of the partial topological group $\rm QS$ for which the neighborhood base is given at the identity by using the quasisymmetry constant and is distributed at every point $g \in \rm QS$ by the right translation. 

The topology on $\rm QS(\mathbb{R})$ is similarly defined.
Conjugation by the Cayley transformation $\varphi(x) =(x-i)/(x+i)$ from $\mathbb{R}$ onto $\mathbb{S}$ 
gives an isomorphism of $\rm QS(\mathbb{R})$ onto $\rm QS$ as a partial topological group.
In particular, 
the characteristic topological subgroup of $\rm QS(\mathbb{R})$,  
denoted by $\rm Sym(\mathbb{R})$, consists of all $\tilde{g} \doteq \varphi^{-1}\circ g\circ\varphi$ for each $g \in \rm Sym$. Equivalently, $\tilde{g}$ can be extended to a quasiconformal homeomorphism of
$\mathbb U$ onto itself with complex dilatation $\mu$ such that 
$$
\inf\,\{\,\Vert\mu|_{\mathbb{U}\setminus K}\Vert_{\infty} \mid K\subset \mathbb{U}:{\rm compact}\} = 0.
$$

Recently, Hu, Wu and Shen \cite{HWS} pointed out that $\rm Sym(\mathbb{R})$ is a nontrivial subset of $\rm S(\mathbb{R})$
(see also Brakalova \cite{Br}). Then,  
we see that $\rm S(\mathbb{R})$ is not a topological group.
Moreover, our main result below implies that $\rm S(\mathbb{R})$ does not even constitute a group.

\begin{theorem}\label{main}
There exists a symmetric homeomorphism $h:\mathbb{R}\to\mathbb{R}$ such that $h^{-1}$ is not symmetric.
\end{theorem}

In this paper, we will give a self-contained proof of Theorem \ref{main}. 
We construct, in Section 3, an example of a symmetric homeomorphism $h$ of $\mathbb{R}$ 
such that the inverse $h^{-1}$ is not symmetric. Before that, in Section 2, we show some general results that help identify the example.
In Section 4, we consider the composition of symmetric homeomorphisms of $\mathbb{R}$, and prove that
the composition does not preserve this class.
Finally in Section 5, we also deal with strongly symmetric homeomorphisms on the real line $\mathbb{R}$, and prove the same results on the inverse and the composition as those for symmetric homeomorphisms of $\mathbb{R}$.

\section{The extension of quasisymmetric functions}
In this section, we prepare certain general arguments for a canonical extension of a function 
to keep quasisymmetry and symmetry, which will be used to identify the function constructed in Section 3 as quasisymmetric and also symmetric. The results in this section also have independent interests of their own. 

\begin{lemma}\label{intermediate}
Let $a,b,c,d>0$ be positive real numbers. Suppose that
$$
{\frac{1}{1+\varepsilon} \leq \frac{b}{a},\ \frac{d}{c} \leq 1+\varepsilon}
$$
for some $\varepsilon \geq 0$.
Then, the following are satisfied:
\begin{enumerate}
\item
$\displaystyle{\frac{1}{1+\varepsilon} \leq \frac{b+d}{a+c} \leq 1+\varepsilon}$.
\item
If $\displaystyle{\frac{r-1}{r+1}a \geq  c}$ and $\displaystyle{\frac{r-1}{r+1} b \geq d}$ for $r > 1$ in addition, then
$\displaystyle{\frac{1}{1+r\varepsilon} \leq \frac{b-d}{a-c} \leq 1+r\varepsilon}$.
\end{enumerate}
\end{lemma}

\begin{proof}
(1) This is well-known as the property of the mediant. 
(2) If $\varepsilon=0$, then the inequalities are clearly satisfied; we may assume that $\varepsilon>0$.
The extra assumption $\frac{r-1}{r+1} a \geq c$ is equivalent to the condition
$$
(1+r\varepsilon -(1-\varepsilon))c \leq (r-1)\varepsilon a.
$$
By $(1+\varepsilon)^{-1} \geq 1-\varepsilon$, this condition implies that
$$
(1+\varepsilon)a-(1+\varepsilon)^{-1}c \leq (1+r\varepsilon)(a-c).
$$
Moreover, the basic assumptions imply that $b \leq (1+\varepsilon)a$ and $d \geq (1+\varepsilon)^{-1}c$.
Thus, we obtain that $b-d \leq (1+r\varepsilon)(a-c)$. By replacing $a$ with $b$ and $c$ with $d$, we also
obtain that $a-c \leq (1+r\varepsilon)(b-d)$.
\end{proof}

\begin{remark}\label{4 points}
{\rm
Let $h$ be an $M$-quasisymmetric function on an interval $I \subset \mathbb R$. For $t>0$,
let $x-t$, $y$, $x$ and $y+t$ be in $I$. We see that
if $x-t \leq y \leq x$, then 
$$
\frac{1}{M} \leq \frac{h(y+t)-h(y)}{h(x)-h(x-t)} \leq M. 
$$
Indeed, we apply Lemma \ref{intermediate} (1) to
$$
\frac{h(y+t)-h(y)}{h(x)-h(x-t)}=\frac{\{h(y+t)-h(\frac{x+y}{2})\}+\{h(\frac{x+y}{2})-h(y)\}}{\{h(\frac{x+y}{2})-h(x-t)\}+\{h(x)-h(\frac{x+y}{2})\}}.
$$
}
\end{remark}

First, we consider the extension of a quasisymmetric function as an odd function.

\begin{proposition}\label{odd}
Let $h(x)$ be a strictly increasing function defined on the interval $[0,L]$ (possibly $L=\infty$) with $h(0)=0$.
Suppose that there are some $\varepsilon \geq 0$ and $\delta>0$ such that 
$$
\frac{1}{1+\varepsilon} \leq m_h(x,t):=\frac{h(x+t)-h(x)}{h(x)-h(x-t)} \leq {1+\varepsilon}
$$
for every $x \in (0,L)$ and for every $t \in (0,\min\{\delta, x, L-x\}]$. 
Let $\widehat h(x)$ be the extension of $h(x)$ to the interval $[-L,L]$ as an odd function satisfying
$\widehat h(-x)=-\widehat h(x)$.
Then,
$$
\frac{1}{1+7 \varepsilon +6 \varepsilon^2+2\varepsilon^3} \leq m_{\widehat h}(x,t):=\frac{\widehat h(x+t)- \widehat h(x)}{\widehat h(x)-\widehat h(x-t)} \leq {1+7 \varepsilon +6 \varepsilon^2+2\varepsilon^3}
$$
is satisfied for every $x \in (-L,L)$ and for every $t \in (0,\min\{\delta, L-|x|\}]$. 
\end{proposition}

\begin{proof}
We divide the choice of $t \in (0,\min\{\delta, L-|x|\}]$ into several cases.
When $t \leq |x|$, $m_{\widehat h}(x,t)$ is clearly bounded by $1+\varepsilon$ and $(1+\varepsilon)^{-1}$.
For the remainder cases, we may assume that $x$ is in $(0,L)$.

(a) Case $t \geq 3x$: We consider the quotient in the form
$$
m_{\widehat h}(x,t)
=\frac{\{h(x+t)-h(\frac{x+t}{2})\}+\{h(\frac{x+t}{2})-h(2x)\}+\{h(2x)-h(x)\}}
{\{h(\frac{x+t}{2})-h(0)\}+\{h(t-x)-h(\frac{x+t}{2})\}+\{h(x)-h(0)\}}.
$$
We remark that $t \geq 3x$ implies that $(x+t)/2 -2x=t-x -(x+t)/2 \geq 0$. 
By Lemma \ref{intermediate} (1), we see that $m_{\widehat h}(x,t)$ is bounded by $1+\varepsilon$ and $(1+\varepsilon)^{-1}$.
Alternatively, we may apply Remark \ref{4 points}.

(b) Case $x \leq t \leq 3x$:
We consider the quotient in the form
$$
m_{\widehat h}(x,t)
=\frac{\{h(x+t)-h(\frac{3x+t}{2})\}+\{h(4x)-h(x)\}-\{h(4x)-h(\frac{3x+t}{2})\}}
{\{h(\frac{3x+t}{2})-h(2x)\}+\{h(x)+h(2x)\}-\{h(\frac{3x+t}{2})-h(t-x)\}}.
$$
We note that Case (a) also implies that
$$
(1+\varepsilon)^{-1} \leq \frac{h(4x)-h(x)}{h(x)+h(2x)}=\frac{\widehat h(4x)-\widehat h(x)}{\widehat h(x)-\widehat h(-2x)} \leq 1+\varepsilon.
$$
Hence, it follows from Lemma \ref{intermediate} (1) that
$$
(1+\varepsilon)^{-1} \leq \frac{\{h(x+t)-h(\frac{3x+t}{2})\}+\{h(4x)-h(x)\}}{\{h(\frac{3x+t}{2})-h(2x)\}+\{h(x)+h(2x)\}} \leq 1+\varepsilon.
$$

By $x \leq t \leq 3x$, we have
conditions
$$
\frac{3x+t}{2} \in [2x,3x], \quad t-x \in [0,2x], \quad x+t \geq \frac{3x+t}{2}.
$$
We use the following estimates:
\begin{align*}
&\quad\
h(4x)-h((3x+t)/2)\\ 
&\leq 
\frac{\{h(4x)-h(3x)\}+\{h(3x)-h(2x)\}}{\{h(4x)-h(3x)\}+\{h(3x)-h(2x)\}+\{h(2x)-h(x)\}}(h(4x)-h(x))\\ 
&\leq
\frac{(1+\varepsilon)+(1+\varepsilon)^2}{1+(1+\varepsilon)+(1+\varepsilon)^2} (h(4x)-h(x))\\
&\leq
\frac{2+3\varepsilon+\varepsilon^2}{3+3\varepsilon+\varepsilon^2} \{(h(x+t)-h((3x+t)/2))+(h(4x)-h(x))\};\\
&\quad\
h((3x+t)/2)-h(t-x)\\
&\leq 
\frac{\{h(\frac{3x+t}{2})-h(2x)\}+\{h(2x)-h(x)\}+\{h(x)-h(0)\}}{\{h(\frac{3x+t}{2})-h(2x)\}+\{h(2x)-h(x)\}+2\{h(x)-h(0)\}}(h((3x+t)/2)+h(x))\\
&\leq 
\frac{\{h(3x)-h(2x)\}+\{h(2x)-h(x)\}+\{h(x)-h(0)\}}{\{h(3x)-h(2x)\}+\{h(2x)-h(x)\}+2\{h(x)-h(0)\}}(h((3x+t)/2)+h(x))\\ 
&\leq
\frac{1+(1+\varepsilon)+(1+\varepsilon)^2}{2+(1+\varepsilon)+(1+\varepsilon)^2} (h((3x+t)/2)+h(x))\\
&=
\frac{3+3\varepsilon+\varepsilon^2}{4+3\varepsilon+\varepsilon^2}\{(h((3x+t)/2)-h(2x))+(h(2x)+h(x))\}.
\end{align*}
Then, applying Lemma \ref{intermediate} (2) for $r=7+6\varepsilon+2\varepsilon^2$, we see that 
$m_{\widehat h}(x,t)$ is bounded by
$1+7 \varepsilon +6 \varepsilon^2+2\varepsilon^3$ and $(1+7 \varepsilon +6 \varepsilon^2+2\varepsilon^3)^{-1}$.
\end{proof}

\begin{corollary}\label{hat}
Let $h$ and $\widehat h$ be the functions as in Proposition \ref{odd}.
If $h$ is quasisymmetric, then so is $\widehat h$. Moreover, if
$h$ is symmetric, then so is $\widehat h$.
\end{corollary}

\begin{proof}
Suppose that there is a constant $M \geq 1$ such that
$M^{-1} \leq m_h(x,t) \leq M$ 
for every $x \in (0,L)$ and for every $t \in (0,\min\{x, L-x\}]$. 
Then, by taking $\varepsilon=M-1$ and $\delta=L$,
we see from Proposition \ref{odd} that $\widehat h$ is quasisymmetric.

Suppose that for every $\varepsilon>0$, there exists some $\delta>0$ such that
$(1+\varepsilon)^{-1} \leq m_h(x,t) \leq 1+\varepsilon$ for every $t \in (0,\min\{\delta,x, L-x\}]$.
We may assume that $\varepsilon \leq 1/8$. Then, Proposition \ref{odd} implies that if 
$t \in (0,\min\{\delta,L-|x|\}]$, then $(1+8 \varepsilon)^{-1} \leq m_{\widehat h}(x,t) \leq 1+8 \varepsilon$.
This shows that $\widehat h$ is symmetric.
\end{proof}

Next, we consider the quasisymmetry of the connection of two quasisymmetric functions.

\begin{proposition}\label{0}
Let $h_-$ and $h_+$ be $M$-quasisymmetric functions on intervals $I_- \subset (-\infty,0]$
and $I_+ \subset [0,\infty)$
respectively, where $I_- \cap I_+=\{0\}$ and $h_-(0)=h_+(0)$. 
Let $h$ be defined 
as $h(x)=h_-(x)$ on $I_-$ and
$h(x)=h_+(x)$ on $I_+$. If there is some $K \geq 1$ such that $h$ satisfies 
$$
\frac{1}{K} \leq \frac{h(t)-h(0)}{h(0)-h(-t)} \leq K
$$
for any $t \in (0,\min\{|I_-|, |I_+|\}]$, then the quasisymmetry quotient satisfies
$C^{-1} \leq m_h(x,t) \leq C$ for any $x \in I_- \cup I_+$ and for any $t>0$
with $x-t,\ x+t \in I_- \cup I_+$,
where $C$ depends only on $K$ and $M$, which can be taken as $C=M(1+M)(K+M)$. 
\end{proposition}

\begin{proof}
We divide a choice of $t \in (0,\min\{|I_-|, |I_+|\}]$ into several cases.
When $t \leq |x|$, $m_{h}(x,t)$ is clearly bounded by $M$ and $M^{-1}$.
For the remainder cases, 
we may assume that $x \in I_+$. The case where $x \in I_-$ is similarly treated.

(a) Case $t \geq 3x$: We consider the quasisymmetry quotient in the form
$$
m_{h}(x,t)
=\frac{\{h(2x)-h(x)\}+\{h(x+t)-h(2x)\}}
{\{h(x)-h(0)\}+\{h(0)-h(x-t)\}}.
$$
Here, by the assumption on the quotient at $0$ and by Remark \ref{4 points} 
(since $t-x \geq 2x$), we have
$$
\frac{h(x+t)-h(2x)}{h(0)-h(x-t)} \leq K\frac{h(x+t)-h(2x)}{h(t-x)-h(0)} \leq KM.
$$
The lower bound is similarly obtained. Then, by Lemma \ref{intermediate} (1), $m_{h}(x,t)$ is bounded by
$KM$ and $(KM)^{-1}$.

(b) Case $x \leq t \leq 3x$: 
We estimate $h(x+t)-h(x)$ and $h(x)-h(x-t)$ from both above and below in terms of $h(x)-h(0)$.
\begin{align*}
h(x+t)-h(x) &\leq (h(2x)-h(x))+(h(3x)-h(2x))+(h(4x)-h(3x)) \\
&\leq (M+M^2+M^3)(h(x)-h(0));\\
h(x+t)-h(x) &\geq (h(2x)-h(x))\\
&\geq M^{-1}(h(x)-h(0));\\
h(x)-h(x-t) &\leq (h(x)-h(0))+(h(0)-h(-x))+(h(-x)-(h(-2x)) \\
&\leq (1+K+KM)(h(x)-h(0));\\
h(x)-h(x-t) &\geq h(x)-h(0).\\
\end{align*}
From these inequalities, we obtain that
$$
\frac{1}{M+KM+KM^2} \leq m_h(x,t) \leq M+M^2+M^3,
$$
from which we have the statement.
\end{proof}

\section{The counter-example}
\subsection{The construction of the example}

For each $n \in \mathbb{N}$, we consider a function $h_n(x)=x^2/(24n)$ on the interval $[1,12n] \subset \mathbb{R}$.
We draw the graph of $y=h_n(x)$ on the $xy$-plane and its $\pi$-rotating copy on the point $O_n=(1,h_n(1))$.
The union of these two curves is denoted by $\mathcal G_n$. Its end points are $E_n=(12n, h_n(12n))$ and
the antipodal point $E'_n$ on the copy. 

We move $\mathcal G_1$ by parallel translation so that $E_1'$ coincides with the origin $(0,0)$ of the $xy$-plane. Next, we move $\mathcal G_2$ by parallel
translation so that
$E'_2=E_1$. We continue this construction for all $n \in \mathbb{N}$; in the positive direction, we 
put each $\mathcal G_n$ from one to another so that $E'_n=E_{n-1}$.
The union $\bigcup_{n=1}^\infty \mathcal G_n$ is denoted by $\mathcal G_+$.
We also make its $\pi$-rotating copy on the origin $(0,0)$, which is denoted by $\mathcal G_-$. Then, we set 
$\mathcal G=\mathcal G_+ \cup \mathcal G_-$.
This curve $\mathcal G$ on the $xy$-plane defines a function $y=h(x)$ for $x \in \mathbb{R}$ that has $\mathcal G$ as its graph.

The end points of $\mathcal{G}_n$ are denoted by
$E_n=(x_n,h(x_n))$ and $E_n'=(x_n',h(x_n'))$ $(x_n >x_n')$. Moreover, the mid point $(x_n'+x_n)/2$ is denoted by $o_n$.
Let the projection of the curve $\mathcal{G}_n$ onto the $x$-axis be the closed interval $G_n=[x_n',x_n]$. 
By the construction, the restriction of $h$ to $G_n$ is $\widehat h_n$, which is given by extending the function $h_n$
on $[o_n,x_n]$ to $[x_n',o_n]$.
The length of the interval $G_n$ is denoted by $|G_n| (=x_n-x_n'=24n-2)$.

Since the derivative $h'_n(12n)$ equals $1$ for every $n \in \mathbb N$, each curve $\mathcal G_n$ has gradient $1$ at the both end points $E_n$ and $E_n'$.
Hence, all pieces are connected smoothly; we see that $h$ is a $C^1$-function
with $h' \neq 0$. Moreover, since $h_n(12n)-h_n(1)=6n-(24n)^{-1}$,
each $\mathcal G_n$ gains $12n-(12n)^{-1}$ in the direction of the $y$-axis. (The gradient of the direction from $E'_n$ to
$E_n$ is nearly $1/2$.) This implies that $h$ is surjective onto $\mathbb{R}$.
Consequently, we have an increasing diffeomorphism $h$ of $\mathbb{R}$ onto $\mathbb{R}$.

We can show that $h$ is (quasisymmetric and moreover) symmetric.
For each function $h_n$, we consider the quasisymmetry quotient
$$
m_{h_n}(x,t):=\frac{h_n(x+t)-h_n(x)}{h_n(x)-h_n(x-t)}=\frac{2x+t}{2x-t} \quad (x \in (1,12n), \ t \in (0,\min\{x-1, 12n-x\}]).
$$
It is clear that $1 < m_{h_n}(x,t) < 3$ for all $n \in \mathbb{N}$.
Moreover, since $x > 1$, we see that $m_{h_n}(x,t) \to 1$ uniformly as $t \to 0$, which is independent of
$x$ and $n$. Then, 
we will extend these estimates from local pieces for $h_n$ 
to the function $h$ on $\mathbb{R}$ globally. 
Necessary arguments to make this rigorous are given in the next subsection. 

The consequence is that there exists a constant $M>0$ so that the quasisymmetry quotient 
$$
m_h(x,t):=\frac{h(x+t)-h(x)}{h(x)-h(x-t)} \quad (x \in (-\infty,\infty), \ t \in (0,\infty))
$$
for $h$ satisfies $1/M \leq m_h(x,t) \leq M$ (which implies that $h$ is quasisymmetric),
and that $m_h(x,t) \to 1$ uniformly as $t \to 0$ for all $x \in \mathbb{R}$.
In particular, $h$ is symmetric.

\subsection{Proof of Theorem 1.1}

We begin with verifying that the inverse function $h^{-1}:\mathbb{R} \to \mathbb{R}$ is not symmetric.
We can also view $\mathcal G$ in the $xy$-plane as the graph of $x=h^{-1}(y)$ by exchanging the roles of $x$ and $y$.
We look at points 
$O_n=(1,h_n(1))$, $A_n=(5,h_n(5))$, and $B_n=(7,h_n(7))$ on each $\mathcal G_n$, 
where $h_n(7)-h_n(5)=h_n(5)-h_n(1)=1/n$.
Let $y_n=h_n(5)$. Then,
$$
\frac{h_n^{-1}(y_n+1/n)-h_n^{-1}(y_n)}{h_n^{-1}(y_n)-h_n^{-1}(y_n-1/n)}=\frac{7-5}{5-1}=\frac{1}{2}
$$
is satisfied for every $n \in \mathbb{N}$.
This implies that the quotient
$$
\frac{h^{-1}(y+s)-h^{-1}(y)}{h^{-1}(y)-h^{-1}(y-s)}
$$ 
does not tend to $1$ uniformly as $s \to 0$. Hence, $h^{-1}$ is not symmetric.

In the remainder of this subsection, we verify that $h$ is symmetric.

\begin{proposition}\label{1}
The quasisymmetry quotient $m_{\widehat h_n}(x,t)$ for each $\widehat h_n$ $(n \in \mathbb N)$ satisfies $m_{\widehat h_n}(x,t) \in [48^{-1},48]$.
Moreover, $m_{\widehat h_n}(x,t) \to 1$ uniformly as $t \to 0$, which depends on neither $x$ nor $n$.
\end{proposition}

\begin{proof}
By parallel translation, we may assume that $o_n=1$.
By simple computation, we have that
$$
m_{h_n}(x,t):=\frac{h_n(x+t)-h_n(x)}{h_n(x)-h_n(x-t)}=\frac{2x+t}{2x-t} \quad (x \in (1,12n), \ t \in (0,\min\{x-1, 12n-x\}]).
$$
Then, the quasisymmetry quotient $m_{h_n}(x,t)$ for each $h_n$ $(n \in \mathbb N)$ satisfies $m_{h_n}(x,t) \in (1, 3)$.
Moreover, $m_{h_n}(x,t) \to 1$ uniformly as $t \to 0$.
Applying Propositions \ref{odd} and \ref{0}, we obtain the statements.
\end{proof}

\begin{proposition}\label{2}
For each $n \in \mathbb N$, the function $y=h(x)$ on $ G_n \cup G_{n+1}$ is quasisymmetric, 
and there exists some constant $M_1$ that does not depend on $n$ such that 
$$
\frac{1}{M_1}\leq m_h(x, t):=\frac{h(x+t)-h(x)}{h(x)-h(x-t)}\leq M_1
$$
for every $x\in G_n\cup G_{n+1}$ and for every $t>0$ with $x-t,\ x+t \in G_n\cup G_{n+1}$.
\end{proposition}

\begin{proof}
For the sake of convenience, we move the curve $\mathcal{G}_n\cup \mathcal{G}_{n+1}$ by parallel translation 
so that $E_n=E_{n+1}'$ coincides with the origin $(0,0)$. 
In this setting, $x_n=0$ and $h(x_n)=0$.
In order to show that $y=h(x)$ is quasisymmetric on $G_n \cup G_{n+1}$ for each $n \in \mathbb N$, 
we use Proposition \ref{0}. It is sufficient to show that there is some $K\geq 1$ such that $h$ satisfies 
$$
\frac{1}{K}\leq \frac{h(t)-h(0)}{h(0)-h(- t)}=\frac{h(t)}{-h(- t)}\leq K
$$
for every $t \in (0, |G_n|]$.

The curve $\mathcal{G}_n\cup \mathcal{G}_{n+1}$ defines a function $y={h}_{n}^{n+1}(x)$ on $[-|G_n|, |G_{n+1}|]$. In particular, 
$$
{h}_{n}^{n+1}(x) = \begin{cases}
\frac{1}{24n}(x+12n)^2 - 6n, &  x \in [-\frac{1}{2}|G_n|,\; 0]\\
-\frac{1}{24(n+1)}(12-x+12n)^2+6(n+1), & x \in [0,\; \frac{1}{2}|G_{n+1}|].
\end{cases}
$$
We see that 
$$
\big|\frac{{h}_{n}^{n+1}(x)}{x} - 1\big|=\begin{cases}
\frac{|x|}{24n}, &  x \in [-\frac{1}{2}|G_n|,\; 0]\\
\frac{x}{24(n+1)}, & x \in [0, \;\frac{1}{2}|G_{n+1}|],
\end{cases}
$$
which implies that
$|\frac{{h}_{n}^{n+1}(x)}{x} - 1| \leq \frac{1}{2}$ if $|x| \leq \frac{1}{2}|G_n|$. 
Namely, 
$$
\frac{|x|}{2} \leq |{h}_{n}^{n+1}(x)| \leq \frac{3|x|}{2} \qquad (x \in [-\frac{|G_n|}{2},\frac{|G_n|}{2}]).
$$
Thus, for any $t \in (0, \frac{1}{2}|G_n|]$, we have that
$$
\frac{1}{3} \leq \frac{{h}_{n}^{n+1}(t)-{h}_{n}^{n+1}(0)}{{h}_{n}^{n+1}(0)-{h}_{n}^{n+1}(-t)}=
\frac{{h}_{n}^{n+1}(t)}{-{h}_{n}^{n+1}(-t)} \leq 3.
$$

For any $t \in (\frac{1}{2}|G_n|,|G_n|]$, we consider $\frac{t}{2} \in (0,\frac{1}{2}|G_n|]$, which satisfies 
\begin{equation}
\tag{$\ast$}
\frac{1}{3} \leq \frac{{h}_{n}^{n+1}(\frac{t}{2})-{h}_{n}^{n+1}(0)}{{h}_{n}^{n+1}(0)-{h}_{n}^{n+1}(-\frac{t}{2})} \leq 3.
\end{equation}
Moreover, by Proposition \ref{1},
we have that
$$
\frac{1}{48} \leq \frac{{h}_{n}^{n+1}(t)-{h}_{n}^{n+1}(\frac{t}{2})}{{h}_{n}^{n+1}(\frac{t}{2})-{h}_{n}^{n+1}(0)} \leq 48;\quad
\frac{1}{48} \leq \frac{{h}_{n}^{n+1}(-t)-{h}_{n}^{n+1}(-\frac{t}{2})}{{h}_{n}^{n+1}(-\frac{t}{2})-{h}_{n}^{n+1}(0)} \leq 48.
$$
We will estimate ${h}_{n}^{n+1}(t)-{h}_{n}^{n+1}(0)$ and ${h}_{n}^{n+1}(0) - {h}_{n}^{n+1}(-t)$ 
from both above and below in terms of ${h}_{n}^{n+1}(\frac{t}{2})-{h}_{n}^{n+1}(0)$ and 
${h}_{n}^{n+1}(0)-{h}_{n}^{n+1}(-\frac{t}{2})$, respectively: 
\begin{equation*}
\begin{split}
& \quad (1+\frac{1}{48})({h}_{n}^{n+1}(\frac{t}{2})-{h}_{n}^{n+1}(0)) \\
&\leq {h}_{n}^{n+1}(t)-{h}_{n}^{n+1}(0)=[{h}_{n}^{n+1}(t)-{h}_{n}^{n+1}(\frac{t}{2})] + [{h}_{n}^{n+1}(\frac{t}{2})-{h}_{n}^{n+1}(0)]\\
&\leq (1+48)({h}_{n}^{n+1}(\frac{t}{2})-{h}_{n}^{n+1}(0));\\
& \quad (1+\frac{1}{48})({h}_{n}^{n+1}(0)-{h}_{n}^{n+1}(-\frac{t}{2})) \\
&\leq {h}_{n}^{n+1}(0)-{h}_{n}^{n+1}(-t)=[{h}_{n}^{n+1}(0)-{h}_{n}^{n+1}(-\frac{t}{2})] + [{h}_{n}^{n+1}(-\frac{t}{2})-{h}_{n}^{n+1}(-t)]\\
& \leq (1+48)({h}_{n}^{n+1}(0)-{h}_{n}^{n+1}(-\frac{t}{2})).
\end{split}
\end{equation*}
Combined with $(*)$, these estimates yield that
$$
\frac{1}{144} \leq \frac{{h}_{n}^{n+1}(t)-{h}_{n}^{n+1}(0)}{{h}_{n}^{n+1}(0)-{h}_{n}^{n+1}(-t)} \leq 144.
$$

We conclude by Proposition \ref{0} that $y=h(x)$ on $ G_n\cup G_{n+1}$ is quasisymmetric, and there exists some constant $M_1 \geq 1$ such that 
$M_1^{-1}\leq m_h(x, t)\leq M_1$
for every $x\in G_n\cup G_{n+1}$ and for every $t >0$ with $x-t,\ x+t \in G_n\cup G_{n+1}$. 
\end{proof}

By applying the reasoning in Proposition \ref{2} repeatedly for several times, we can obtain the following.
\begin{corollary}\label{s}
For any $n \in \mathbb N$, the function $y=h(x)$ on $\bigcup_{i=0}^{s}G_{n+i}$ is quasisymmetric, 
and there exists some constant $M_2 \geq 1$ 
depending only on $s$ such that 
$$
\frac{1}{M_2}\leq m_h(x, t)\leq M_2
$$
for every $x\in \bigcup_{i=0}^{s}G_{n+i}$ and for every $t>0$ with $x-t,\ x+t \in \bigcup_{i=0}^{s}G_{n+i}$.
\end{corollary}

\begin{proposition}\label{big}
The quasisymmetry quotient $m_h(x,t)$ for $h$ satisfies
$$
\frac{1}{8} \leq m_h(x,t) \leq 8
$$
for every $x \in G_n$ $(n \in \mathbb N)$ and for every $t \in [\sum_{i=0}^{3}|G_{n+i}|, x)$.
\end{proposition}
\begin{proof}
For every $x\in\mathbb{R}$ and for every $t>0$, we have $h(x+t)-h(x) \leq t$ and $h(x)-h(x-t)\leq t$.
In what follows, we will estimate $h(x+t)-h(x)$ and $h(x)-h(x-t)$ from below by $\frac{t}{8}$.
This yields the statement.

For any $n\geq 1$ and $k \geq 3$, we see that $\sum_{i=1}^{k}|G_{n+i}| \geq |G_{n+k+1}|$.
Then,  
for every $x \in G_n$ and for any $t \in [\sum_{i=0}^{3}|G_{n+i}|, \infty)$,
there exists some integer $N_1 \geq n+3$ such that $x+t \in G_{N_1}$ and 
$$
x \leq x_n \leq x+\frac{t}{4} \leq x+\frac{t}{2} \leq x_{N_1}'\leq x+t .
$$
For any $n\geq 1$, the gradient of the direction from $E_n'$ to $E_n$ is greater than $\frac{1}{2}$.
Then, we conclude that 
$$
h(x+t)-h(x) \geq h(x_{N_1}')-h(x_n) \geq (x_{N_1}'-x_n)\times \frac{1}{2} \geq \frac{t}{8}.
$$

Noting that $x-t \geq 0$, we similarly obtain that
there exists some integer $N_2 \leq n-3$ such that $x-t \in G_{N_2}$ and 
$$
x-t \leq x_{N_2} \leq x-\frac{t}{2} \leq x-\frac{t}{4} \leq x_{n}'\leq x.
$$
This implies that  
$$
h(x)-h(x-t) \geq h(x_{n}')-h(x_{N_2}) \geq (x_{n}'-x_{N_2})\times \frac{1}{2}
\geq \frac{t}{8},
$$
which completes the proof.
\end{proof}

The following main result follows from Corollaries \ref{hat} and \ref{s}, and Proposition \ref{big}.
\begin{corollary}
The $h$ is quasisymmetric on the real line $\mathbb{R}$.
\end{corollary}

\begin{proof}
By Corollary \ref{s} for a sufficiently large $s$ and Proposition \ref{big}, we see that $h(x)$ is quasisymmetric on $x \geq 0$.
Then by Corollary \ref{hat}, $h$ is quasisymmetric on the real line $\mathbb{R}$.
\end{proof}

\begin{theorem}\label{sym}
The $h$ is symmetric on the real line $\mathbb{R}$.
\end{theorem}
\begin{proof}
We first consider the case where $x$ is close to some $x_n$ $(n \in \mathbb{N})$. In order to estimate the quasisymmetry
quotient $m_h(x,t)$ when $t$ is sufficiently small,  
we also move the curve $\mathcal{G}_n\cup\mathcal{G}_{n+1}$ by parallel translation as in the proof of Proposition \ref{2} so that $E_n=E_{n+1}'$ coincides with the origin $(0,0)$.
Then, the function $y={h}_{n}^{n+1}(x)$ becomes to be defined for $x\in[-|G_n|, |G_{n+1}|]$.
Its derivative satisfies that
$$
\big|({h}_{n}^{n+1})'(x) - ({h}_{n}^{n+1})'(0)\big|=\big|({h}_{n}^{n+1})'(x) - 1\big|=\begin{cases}
\frac{|x|}{12n}, &  x \in [-\frac{1}{2}|G_n|,\; 0]\\
\frac{x}{12(n+1)}, & x \in [0, \;\frac{1}{2}|G_{n+1}|].
\end{cases}
$$
It follows that if we choose $\delta=\frac{12\epsilon}{2+\epsilon}$ for any $\epsilon > 0$, then 
$$
\big|({h}_{n}^{n+1})'(x) - 1\big|\leq \frac{\epsilon}{2+\epsilon}
$$
for every $x\in[-\delta, \delta]$. This implies that 
$$
(1+\epsilon)^{-1}\leq \frac{({h}_{n}^{n+1})'(x_{\alpha})}{({h}_{n}^{n+1})'(x_{\beta})}\leq 1+\epsilon
$$
for any $x_{\alpha}, x_{\beta} \in [-\delta, \delta]$. 
Thus, for every $x\in [-\frac{\delta}{2}, \frac{\delta}{2}]$ and for every $t \in (0,\frac{\delta}{2})$,
there exist some $x_{\alpha} \in (x, x+t) \subset [-\delta, \delta]$ and $x_{\beta} \in (x-t, x) \subset [-\delta, \delta]$ such that 
$$
(1+\epsilon)^{-1} \leq \frac{{h}_{n}^{n+1}(x+t)-{h}_{n}^{n+1}(x)}{{h}_{n}^{n+1}(x)-{h}_{n}^{n+1}(x-t)}=\frac{({h}_{n}^{n+1})'(x_{\alpha})}{({h}_{n}^{n+1})'(x_{\beta})}\leq 1+\epsilon.
$$ 

Since the above estimate is independent of $n \in \mathbb N$,
we have proved that for any $n \in \mathbb N$ and $\epsilon >0$,
every $x\in [x_n-\frac{\delta}{2}, x_n+\frac{\delta}{2}]$ and every $t \in(0,\frac{\delta}{2})$ satisfy 
that
$$
(1+\epsilon)^{-1} < \frac{h(x+t)-h(x)}{h(x)-h(x-t)}< 1+\epsilon.
$$
By the same reasoning as above, we 
see that for every $x \in[0, \frac{\delta}{2}]$ and for every $t \in (0,\frac{\delta}{2})$, the last inequality still holds.
In the other case where $x \in (x_n' + \frac{\delta}{2}, x_n-\frac{\delta}{2})$ for some $n \in \mathbb N$,
it follows from Proposition \ref{1} that 
$$
m_h(x,t)=\frac{h(x+t)-h(x)}{h(x)-h(x-t)} \to 1
$$ 
uniformly as $t \to 0$. 
Thus, we obtain that $h(x)$ is symmetric on $x \geq 0$.
Then, Corollary \ref{hat} 
completes the proof by showing that $h$ is symmetric on $\mathbb R$.
\end{proof}

\section{A remark on the composition}

We have seen that ${\rm S}(\mathbb R)$ is not a subgroup of ${\rm QS}(\mathbb R)$
by showing that the inverse $h^{-1}$ for $h \in {\rm S}(\mathbb R)$ does not necessarily belong to ${\rm S}(\mathbb R)$.
We can also show that the composition of elements in ${\rm S}(\mathbb R)$
does not necessarily belong to ${\rm S}(\mathbb R)$ either.

\begin{theorem}\label{composition}
There exist symmetric homeomorphisms $g$ and $h$ on the real line $\mathbb R$ such that $h \circ g$ is not symmetric.
\end{theorem}

\begin{proof}
We define the following function on $\mathbb R$:
$$
g(x) = \begin{cases}
(x+1)^2-1, &  x \geq 0\\
-(x-1)^2+1, & x \leq 0.
\end{cases}
$$
It is easy to see that $g$ is symmetric by a simpler argument than before. 
We use the same symmetric homeomorphism $h$ and the notation as in Section 3.
Let $c_n=g^{-1}(o_n+1)$ and 
$t_n=g^{-1}(o_n+1)-g^{-1}(o_n)>0$ for each $n \in \mathbb N$.
It is clear that $t_n \to 0$ as $n \to \infty$. Then, we
consider the quasisymmetry quotient
$$
m_{h \circ g}(c_n,t_n)=\frac{h \circ g(c_n+t_n)-h \circ g(c_n)}{h \circ g(c_n)-h \circ g(c_n-t_n)}
=\frac{h \circ g(c_n+t_n)-h(o_n+1)}{h(o_n+1)-h(o_n)}.
$$
Since $g(c_n+t_n) \geq o_n+2$, we have that
$$
m_{h \circ g}(c_n,t_n)
\geq \frac{h(o_n+2)-h(o_n+1)}{h(o_n+1)-h(o_n)}=\frac{9-4}{4-1}>1
$$ 
for every $n \in \mathbb N$.
This shows that $h \circ g$ is not symmetric.
\end{proof}

\begin{remark}\label{simplify}
{\rm
To avoid the complicated construction of $h$ in Section 3,
we may simplify $h$ by replacing all $h_n$ $(n \in \mathbb N)$ with $h_1$.
This is enough for the purpose of Theorem \ref{composition} though
the inverse $h^{-1}$ is still symmetric in this case.
}
\end{remark}

\section{Strongly symmetric homeomorphisms on $\mathbb{R}$}

In this section, we discuss another subclass of quasisymmetric homeomorphisms, which we call strongly symmetric homeomorphisms. We will
show that the composition and the inverse do not preserve this class by using our constructions in Sections 3 and 4.

We recall the notion of strongly quasisymmetric homeomorphism in the sense of Semmes \cite{Se}. An increasing homeomorphism $h$ of the real line $\mathbb{R}$ onto itself is said to be {\it strongly quasisymmetric} if there exist two positive constants $C_1$ and $C_2$ such that 
$$
\frac{|h(E)|}{|h(I)|}\leq C_1\bigg(\frac{|E|}{|I|}\bigg)^{C_2}
$$
whenever $I\subset \mathbb{R}$ is a bounded interval and $E\subset I$ a measurable subset. Equivalently, $h$ is strongly quasisymmetric if and only if $h$ is locally absolutely continuous so that $h'$ belongs to the class of weights $A^{\infty}$ introduced by Muckenhoupt (see \cite{CF}, \cite{Ga}). In particular, $\log h'$ belongs to $\rm BMO(\mathbb{R})$, the space of locally integrable functions on  $\mathbb{R}$ of bounded mean oscillation (see \cite{Ga}). Let $\rm SQS(\mathbb{R})$ denote the set of all strongly quasisymmetric homeomorphisms of $\mathbb{R}$ onto itself. It is known that $\rm SQS(\mathbb{R})$ is a subgroup of $\rm QS(\mathbb{R})$. Due to the conformal invariance of strongly quasisymmetric homeomorphisms, we have a parallel notion of strongly quasisymmetric homeomorphisms of the unit circle to that of the real line which was developed in the papers \cite{AZ} and \cite{SW}.

Now we say that a strongly quasisymmetric homeomorphism $h \in \rm SQS(\mathbb{R})$ is {\it strongly symmetric} if $\log h'$ belongs to $\rm VMO(\mathbb{R})$, the space of locally integrable functions on $\mathbb{R}$ of vanishing mean oscillation. Let $\rm SS(\mathbb{R})$ denote the set of all strongly symmetric homeomorphisms of the real line $\mathbb{R}$. 
This class was first introduced in Shen \cite{Sh19} during his study of the VMO Teichm\"uller space on the real line. 
In particular, it was proved that if an increasing homeomorphism $h$ of $\mathbb R$ onto itself can be extended to a quasiconformal homeomorphism of $\mathbb U$ onto itself whose Beltrami coefficient $\mu$ 
induces a vanishing Carleson measure on $\mathbb U$ then $h$ is strongly symmetric. 
It is known that $\rm SS(\mathbb{R}) \subset \rm S(\mathbb{R})$,
which we may obtain by examining the proof of 
\cite[Lemma 3.3]{Sh} in the unit circle case. Especially, this implies that both $h^{-1}$ and $h\circ g$ are not in $\rm SS(\mathbb{R})$ by Theorems \ref{main} and \ref{composition}, where $h$ and $g$ are symmetric homeomorphisms constructed in Sections 3 and 4, respectively. 

In the remainder of this section, we prove that both the symmetric homeomorphisms $h$ and $g$ are in $\rm SS(\mathbb{R})$. This shows that neither the composition nor the inverse preserves this class.

\begin{remark}
{\rm
We propose the subclass $\rm SS(\mathbb{R})$ as a natural counterpart to  the unit circle case $\rm SS(\mathbb{S})$, which denotes the set of all absolutely continuous sense-preserving homeomorphisms $h$ of the unit circle $\mathbb{S}$ onto itself with $\log h'\in \rm VMO(\mathbb{S})$. 
This class was introduced in \cite{Pa} when Partyka studied eigenvalues of quasisymmetric automorphisms determined by VMO functions. It was investigated further later in \cite{SW} and \cite{We} during their study of BMO theory of the universal Teichm\"uller space. In particular, it was proved that $\rm SS(\mathbb{S})$ is a subgroup of $\rm SQS(\mathbb{S})$, the group of all strongly quasisymmetric homeomorphisms of $\mathbb S$. Since the logarithmic derivative does not have conformal invariance, the notion of strongly symmetric homeomorphism on the real line is different from the one on the unit circle. In fact, ${\rm SS}(\mathbb R) \neq \gamma^{-1} {\rm SS}'(\mathbb S) \gamma$
for the Cayley transformation $\gamma: \mathbb R \to \mathbb S \setminus \{1\}$ defined by $\gamma(x)=(x-i)/(x+i)$, where
${\rm SS}'(\mathbb S)$ denotes the subgroup of ${\rm SS}(\mathbb S)$ consisting of all elements that fix $1 \in \mathbb S$. The detail will appear in our forthcoming paper. 
}
\end{remark}

\begin{lemma}\label{gSQS}
The $g$ is strongly quasisymmetric on the real line $\mathbb{R}$.
\end{lemma}
\begin{proof}
It suffices to show that there exists a constant $C$ such that 
$$
\frac{|g(E)|}{|g(I)|}\leq C\frac{|E|}{|I|}
$$
for each bounded interval $I\subset \mathbb{R}$ and any measurable subset $E\subset I$.

We divide the choice of an interval $I \subset \mathbb{R}$ and a measurable subset $E\subset I$ into several cases. For $I \subset \mathbb{R}$ and $E\subset I$, we set $I \cap \mathbb{R}_+ = I_+$, $I \cap \mathbb{R}_-= I_-$, $E \cap \mathbb{R}_+ = E_+$, and 
$E \cap \mathbb{R}_- = E_-$, where $\mathbb{R}_+=[0,+\infty)$ and $\mathbb{R}_-=(-\infty,0]$.

Case (a) $I \subset \mathbb{R}_+$: Set $I = [a, b]$. Then we have 
$$
\frac{|g(E)|}{|g(I)|} =\frac{\int_E 2(x+1)dx}{\int_I 2(x+1)dx}= \frac{|E|}{|I|}\frac{2+\int_E 2x dx/|E|}{2+\int_I 2x dx/|I|} \leq \frac{|E|}{|I|} \frac{2+2b}{2+a+b} < 2 \frac{|E|}{|I|}. 
$$

Case (b) $I \subset \mathbb{R}_-$: This is similarly treated to Case (a).

Case (c) $0 < |I_-| \leq |I_+|$: (c1) If $|E_-| = 0$, then we can conclude by Case (a) that 
$$
\frac{|g(E)|}{|g(I)|} = \frac{|g(E)|}{|g(I_-)| + |g(I_+)|} \leq \frac{|g(E)|}{|g(I_+)|} < 2 \frac{|E|}{|I_+|} \leq 4 \frac{|E|}{|I|}. 
$$
(c2) If $|E_+| = 0$, by (c1) and the fact $|g(E)| = |g(-E)|$ for $-E = \{x \mid -x \in E\}$, we have 
$$
\frac{|g(E)|}{|g(I)|} = \frac{|g(-E)|}{|g(I)|} < 4 \frac{|E|}{|I|}. 
$$
(c3) If $|E_-| > 0$ and $|E_+| > 0$, then it follows from (c1) and (c2) that
$$
\frac{|g(E)|}{|g(I)|} = \frac{|g(E_-)|}{|g(I)|} + \frac{|g(E_+)|}{|g(I)|} < 4 \frac{|E_-|}{|I|} + 4 \frac{|E_+|}{|I|} = 4 \frac{|E|}{|I|}. 
$$

Case (d) $0 < |I_+| < |I_-|$: This is similarly treated to Case (c). 
\end{proof}

\begin{lemma}\label{hSQS}
The $h$ is strongly quasisymmetric on the real line $\mathbb{R}$.
\end{lemma}

\begin{proof}
Let $I$ be any bounded interval in $\mathbb R$. 
Since $h$ is extended to 
the negative real line $(-\infty,0]$ as an odd function, we set 
$G_{-n}=-G_n$ and $o_{-n}=-o_n$ for every $n \in \mathbb N$.
We first consider the case that $I$ is contained in $G_n$ for some 
$n \in \mathbb N \cup (-\mathbb N)$. 
In this case, we move the domain $G_n$ of $h$ by translation so that the mid point $o_n$ of $G_n$ coincides with $0$. In this transformation, we can represent $h$ as 
$$
h(x) =\begin{cases} \frac{-1}{24|n|} (x-1)^2 + \frac{1}{24|n|}   , & x\in [-\frac{1}{2}|G_n|, \;0]\\
\frac{1}{24|n|} (x+1)^2 - \frac{1}{24|n|} , & x\in [0,\;\frac{1}{2}|G_n|].
\end{cases}
$$ 
By the same proof as in Lemma \ref{gSQS}, we can obtain 
$$
\frac{|h(E)|}{|h(I)|}  < 4 \frac{|E|}{|I|}
$$
for each interval $I \subset G_n$ and any measurable subset $E \subset I$. 

It remains to consider the case that $I$ intersects plural intervals $G_n$. For each $n \in \mathbb N \cup (-\mathbb N)$,  we separate $G_n$ into two parts:
$$
G_n^{\flat} = \{x \in G_n \mid h'(x) \leq \frac{1}{2}\}
;\quad
G_n^{\sharp} = \{x \in G_n \mid h'(x) > \frac{1}{2}\}. 
$$
Here, $G_n^{\flat}$ is a symmetric sub-interval of $G_n$ with respect to $o_n$ with length less than $\frac{1}{2}|G_n|$, 
and $G_n^{\sharp} = G_n - G_n^{\flat}$ is of length greater than $\frac{1}{2}|G_n|$. The union 
$\bigcup G_n^{\flat}$ taken over all $n$ is denoted by $G^{\flat}$ and $\bigcup G_n^{\sharp}$ is denoted by $G^{\sharp}$. Let
$I^{\flat} = I \cap G^{\flat}$ and $I^{\sharp} = I \cap G^{\sharp}$. By the assumption that $I$ intersects plural intervals $G_n$,
we see that $\frac{1}{2}|I^{\flat}| < |I^\sharp|$.

Since $0< h'(x) \leq 1$ for any $x \in \mathbb{R}$, we have $|h(E)| \leq |E|$
for each measurable subset $E\subset I$. 
Noting that $|I| = |I^{\flat}| + |I^{\sharp}|$ and $\frac{1}{2}|I^{\flat}| < |I^\sharp|$, we have $|I^\sharp| > \frac{1}{3}|I|$. Then, 
$$
|h(I)| \geq |h(I^\sharp)|=\int_{I^\sharp}h'(x)dx > \frac{1}{2}|I^\sharp| > \frac{1}{6}|I|. 
$$
Thus, 
$$
\frac{|h(E)|}{|h(I)|}  < 6 \frac{|E|}{|I|},
$$
which completes the proof. 
\end{proof}

We have seen that both $g$ and $h$ in ${\rm S}(\mathbb R)$ constructed in Sections 3 and 4 are 
strongly quasisymmetric homeomorphisms. We will show that they are in fact strongly symmetric.

\begin{theorem}\label{hSS}
The $h$ is strongly symmetric on the real line $\mathbb{R}$.
\end{theorem}

\begin{proof}
In order to prove $\log h' \in \rm VMO(\mathbb{R})$, 
it suffices to show that $\log h'$ is a uniformly continuous function in $\rm BMO(\mathbb{R})$. 
In fact, letting $\rm UC$ denote the set of all uniformly continuous functions on $\mathbb{R}$, we have that
$\rm VMO(\mathbb{R})$ is the closure of $\rm UC \cap BMO(\mathbb{R})$ under the BMO norm, as is shown 
in \cite[Theorem 5.1]{Ga} (see also \cite{Sa}). 

To prove that $\log h'$ is uniformly continuous, 
it is enough to show that $\log h'_n(x)=\log (x/12n)$ on $[1,12n]$ for each $n \in \mathbb N$ is uniformly continuous
independent of $n$. For any $x, x' \geq 1$, we have
$$
|\log h'_n(x)-\log h'_n(x')|=|\log (x/12n)-\log (x'/12n)|=|\log x-\log x'|.
$$
Hence, the modulus of the continuity of $\log h'_n$ is the same as $\log x$ $(x \in [1,12n])$ for all $n$.
This implies that $\log h'$ is uniformly continuous. Since $h$ is strongly quasisymmetric on the real line $\mathbb{R}$ by Lemma \ref{hSQS}, 
we have $\log h' \in \rm BMO(\mathbb{R})$ (see \cite{Ga}). This proves the theorem.  
\end{proof}

\begin{theorem}\label{gSS}
The $g$ is strongly symmetric on the real line $\mathbb{R}$.
\end{theorem}
\begin{proof}
By simple computation, we have
$$
\log g'(x) = \begin{cases}
\log 2 + \log (x+1), &  x \geq 0\\
\log 2 + \log (1-x), & x \leq 0.
\end{cases}
$$
Then, $\log g'$ is an even function and uniformly continuous on the real line $\mathbb{R}$. Since $g$ is strongly quasisymmetric by Lemma \ref{gSQS}, we have  $\log g'$ is in $\rm BMO(\mathbb{R})$. Thus, $\log g' \in \rm VMO(\mathbb{R})$. 
\end{proof}

The main result of this section follows readily.
\begin{corollary}
The subclass $\rm SS(\mathbb{R})$ is not necessarily preserved under the inverse and the composition.
\end{corollary}
\begin{proof}
We know that $h$ and $g$ are in $\rm SS(\mathbb{R}) \subset {\rm S}(\mathbb{R})$ by
Theorems \ref{hSS} and \ref{gSS}, but $h^{-1} \notin \rm S(\mathbb{R})$ and $h\circ g \notin \rm S(\mathbb{R})$
by Theorems \ref{main} and \ref{composition}. 
\end{proof}

\begin{remark}
{\rm
(1) In order to consider the composition $h \circ g$, we may replace $h$ with the simpler symmetric homeomorphism as in Remark \ref{simplify}.
Since $\log h'$ for this simplified $h$ is bounded, 
by \cite[Lemma 1.4]{Pa}, $h$ belongs to $\rm SQS(\mathbb{R})$. 
It is easy to see that $\log h'$ belongs to ${\rm VMO}(\mathbb R)$. 
Hence, $h \in {\rm SS}(\mathbb R)$ in this case. (2) We verified that
the previous $h$ belongs to ${\rm SS}(\mathbb R)$ in Theorem \ref{hSS}. If we use the inclusion relation ${\rm SS}(\mathbb R) \subset {\rm S}(\mathbb R)$
based on BMO theory, this gives an alternative proof of the fact that the $h$ belongs to ${\rm S}(\mathbb R)$,
which has been shown in the proof of Theorem \ref{main}.
(3) In our forthcoming paper, we will present a proof of the fact that the simplified $h$ and the $g$ have 
quasiconformal extensions to the upper half-plane $\mathbb U$ whose complex dilatations induce 
vanishing Carleson measrures on $\mathbb U$.
}
\end{remark}

\end{document}